\documentclass[a4paper,reqno]{amsart}

\usepackage{graphicx}
\usepackage{mathrsfs}
\usepackage{amsfonts}
\usepackage{amssymb}
\usepackage{amsmath}
\usepackage{amsthm}
\usepackage{amscd}
\usepackage[all,2cell]{xy}
\usepackage{color}
\usepackage[pagebackref,colorlinks]{hyperref}
\usepackage{enumerate}

\usepackage{comment}
\usepackage{hyperref}
\UseAllTwocells \SilentMatrices

\numberwithin{equation}{section}

\usepackage{comment}

\theoremstyle{plain}
\newtheorem{thm}{Theorem}[section]
\newtheorem{prop}[thm]{Proposition}
\newtheorem{cor}[thm]{Corollary}
\newtheorem{lemma}[thm]{Lemma}

\theoremstyle{definition}
\newtheorem{deff}[thm]{Definition}

\theoremstyle{remark}

\def\g{\gamma}

\def\G{\Gamma}
\def\mG{\mathcal{G}}
\def\xra{\xrightarrow[]{}}
\def\a{\alpha}
\def\b{\beta}

\def\sub{\subseteq}

\def \Z{\mathbb Z}
\def\-{\text{-}}

\newcommand{\End}{\operatorname{End}}

\newcommand{\Hom}{\operatorname{Hom}}

\newcommand{\supp}{\operatorname{supp}}

\newcommand{\MOD}{\operatorname{MOD}}

\begin{document}

\title[Simple flat Leavitt path algebras are regular]{Simple flat Leavitt path algebras are von Neumann regular}

\author{A.A. Ambily}

\address{Cochin University of Science and Technology, India\\ Current address: Western Sydney University\\
Australia}

\email{ambily@cusat.ac.in, a.ambattuasokan@westernsydney.edu.au}

\author{Roozbeh Hazrat}
\address{
Western Sydney University\\
Australia} \email{r.hazrat@westernsydney.edu.au}

\author{Huanhuan Li}

\address{
Western Sydney University\\
Australia} \email{h.li@westernsydney.edu.au}



\date{\today}

\begin{abstract}  
For a unital ring,  it is an open question whether flatness of simple modules implies all modules are flat and thus the ring is von Neumann regular. 
The question was raised by Ramamurthi over 40 years ago~\cite{ram} who called such rings SF-rings (i.e., simple modules are flat). 
In this note we show that a SF Steinberg algebra of an ample Hausdorff groupoid, graded by an ordered group, has an aperiodic unit space. For graph groupoids this implies that the graphs are acyclic. Combining with the Abrams-Rangaswamy Theorem~\cite{abramsranga}, it follows that SF Leavitt path algebras are regular, answering Ramamurthi's question in positive for the class of Leavitt path algebras. 
\end{abstract}

\maketitle

\begin{flushright}
\emph{On the occasion of his 80th birthday\\ to Kulumani M. Rangaswamy\\ whose passion for Mathematics is contagious } 
\end{flushright}

\section{Introduction}

There is a substantial amount of literature on characterising a ring in terms of certain properties of its simple modules such as \cite{lam,Vil}. In 1975, Ramamurthi~\cite{ram} considerd left (right) SF-rings whose left (right) simple modules are flat. It is known that a ring with unit is von Neumann regular if and only if all modules are flat~\cite[Corollary 1.13]{goodearlbook}. In ~\cite{ram}, Ramamurthi raised a question whether flatness of all left (right) simple modules implies that the ring is von Neumann regular. Since then, SF-rings have been extensively studied by many authors and the regularity of SF-rings which satisfy some certain additional conditions is established (see for example~\cite{rege, xiao}). 
The question whether SF-rings are von Neumann regular is still open.

In the last decade, Leavitt path algebras of graphs were introduced as an algebraisation of graph $C^*$-algebras \cite{lpabook}. Quite recently Steinberg algebras were introduced as an algebraisation of the groupoid $C^*$-algebras \cite{cfst, st}. Steinberg algebras include Leavitt and inverse semigroup algebras (see \cite{cs,st}). These classes of algebras have been attracting significant attention.


Since Leavitt and Steinberg algebras are rings with local units, in this note, we first generalise the concept of left SF-rings to the setting of rings with local units. We will show that if a Steinberg algebra of a graded ample Hausdorff groupoid, graded by an ordered group, with coefficients in a field is a left SF-ring, then the unit space of the groupoid is aperiodic (see Proposition \ref{hyhygt}). Recall that the Steinberg algebra of a graph groupoid is isomorphic to the Leavitt path algebra of the graph \cite{cs}. The aperiodicity of unit space in this setting implies that the graph is acyclic. Combining this with the Abrams-Rangaswamy Theorem~\cite{abramsranga}, it follows that SF Leavitt path algebras are von Neumann regular rings, thus answering Ramamurthi's question in positive for the class of Leavitt path algebras.

The structure of this note is as follows. In Section \ref{section2} we generalise the concept of left SF-ring to the setting of rings with local units and extend equivalent statements for von Neumann regular unital rings to rings with local units. In Section \ref{section3}, we first recall concepts of ample groupoids and Steinberg algebras of ample Hausdorff groupoids. Then in subsection \ref{subsection33}, we give a certain necessary condition for a Steinberg algebra of an ample Hausdorff groupoid to be von Neumann regular. In subsection \ref{subsection34} we prove that for a $\G$-graded ample Hausdorff groupoid $\mG$ (i.e.,  the groupoid $\mG$ equipped with a continuous $1$-cocycle $c:\mG\xra \G$), if its Steinberg algebra is a left SF-ring, then the unit space $\mG^{(0)}$ is aperiodic. The conclusion in the case of Leavitt path algebras then follows.

\section{Simple flat and von Neumann regular rings with local units}\label{section2}

 In this section, we generalise the concept of unital left (right) SF-rings and von Neumann regular rings to the setting of rings with local units and 
 extend their characterisations into this setting.

The majority of calculus of Hom and tensor product in the setting of unital rings can be extended analogously to the setting of rings with local units. However the concept of left free modules should be replaced by left quasi-free modules (i.e., a direct sum of principal left ideals generated by idempotents). Therefore extending statements related to flatness from unital to local unital rings, although similar, requires a bit of care. For this reason in this section we shall establish some of the basic facts that we later need in the setting of rings with local units.

\subsection{Simple flat rings}
Let $A$ be an associative ring (not necessarily with unit). $A$ is called \emph{a ring with local units} if for any finitely many elements $a_1, \cdots, a_n\in A$, there exists an idempotent $e\in A$ such that $\{a_1, \cdots, a_n\}\subseteq eAe$.  The book of Wisebaur~\cite[Chapter~10]{wisbauer} is an excellent source for the treatment of rings with local units.

Let $A$ be a ring with local units. A left $A$-module $M$ is called \emph{unital} if $AM=M$. We denote by $A\-\MOD$ the category of all the left unital $A$-modules and by $\MOD\- A$ the category of all the right unital $A$-modules. 


The calculus of tensor products for rings with local units is similar to the case of unital rings. We demonstrate this with one example which we use in the paper. 

\begin{prop}\label{hytrewe}
Let $A$ be a ring with local units, $M$ a unital right $A$-module and $I$ a left ideal of $A$. Then the natural map 
\begin{align*}
\phi:M\otimes_A A/I &\longrightarrow M/MI,\\ 
m\otimes I+r&\longmapsto MI+mr
\end{align*}
is a group isomorphism. 
\end{prop}
\begin{proof}
Note that the map $M \times A/I \rightarrow M/MI; (m,I+r)\mapsto MI+mr$ is an $A$-balanced map inducing the group homomorphism $\phi$. Since $MA=M$, $\phi$ is an epimorphism. If $\phi(\sum m_i\otimes I+r_i)=0$. Then $\sum_i m_ir_i \in MI$. Thus $\sum_i m_ir_i=\sum_j m'_js_j$, where $m_j'\in M$ and $s_j\in I$.  Since $A$ has local units, there is an idempotent $e\in A$ such that $er_i=r_ie=r_i$, $es_j=s_je=s_j$ for all $i,j$'s. Therefore 
\begin{multline*}
\sum m_i\otimes I+r_i=\sum m_i\otimes r_i(I+e)=\sum m_ir_i\otimes I+e=\sum m'_j s_j\otimes I+e=\sum m'_j \otimes I+s_je=\sum m'_j \otimes I+s_j=0.
\end{multline*} 
\end{proof}

A left $A$-module $N$ is called \emph{flat} in $A\-\MOD$ if the functor $-\otimes_A N: \MOD \-A \xra AB$ is exact, where $AB$ is the category of abelian groups. Recall from \cite[\S49.1]{wisbauer} that for $A$ a ring with local units,
\begin{itemize}
\item[(i)] $A$ is flat and a generator in $A\-\MOD$;
\item[(ii)] if $N\in A\-\MOD$, then for finitely many $n_1, \cdots, n_k \in N$ there exists an idempotent $e\in A$ with $en_i = n_i$ for $i = 1, \cdots, k$.
\end{itemize}

The following lemma gives characterisations of flat modules in $A\-\MOD$, which is proved in \cite[\S49.5]{wisbauer} and \cite[\S36.2]{wisbauer}.

\begin{lemma} \label{charflat} Let $A$ be a ring with local units. The following conditions are equivalent:

\begin{itemize}
\item[(1)] A left $A$-module $N$ is flat in $A\-\MOD$;
 
\item[(2)] The functor  $-\otimes_A N$ is exact on exact sequences of the form
$0\xra J_A\xra A_A$ (with $J_A$ finitely generated right ideal of $A$);

 \item[(3)] $J \otimes_A N \xra JN, i\otimes n\mapsto in$, is injective (an isomorphism) for every (finitely generated) right ideal $J\subseteq A$;

\item[(4)] $N$ is a direct limit of projective modules in $A\-\MOD$;

\item[(5)] $N$ is a direct limit of finitely generated projective
modules in $A\-\MOD$. 
\end{itemize} 
\end{lemma}

We are in a position to define the simple flat rings in the setting of rings with local units.

\begin{deff}\label{sfring} Let $A$ be a ring with local units. We call $A$ a \emph{left simple flat ring} (left SF-ring for short) if all simple left $A$-modules are flat in  $A\-\MOD$. The right SF-rings are defined analogously.  
\end{deff}

We need an element-wise characterisation of SF-rings with local units (Theorem~\ref{sfchar}). For this we have to extend some results from the unital setting to rings with local units. 

\begin{lemma} \label{iff}Let $A$ be a ring with local units. Suppose that $V$ is a left flat $A$-module and $0\xra K\xra V\xra V'\xra 0$ is a short exact seqence in $A\-\MOD$. Then $V'$ is flat if and only if for each (finitely generated) right ideal $I$ of $A$, $IK=K\cap IV$.
\end{lemma}

\begin{proof} By Lemma \ref{charflat} (1) and (3), it suffices to prove the following statement: $I \otimes_A V' \xra IV', i\otimes v\mapsto iv$ is an isomorphism for every right ideal $I\subseteq A$ if and only if $IK=K\cap IV$ for every right ideal $I\subseteq A$. The proof for the statement is similar to the case of unital rings; compare \cite[19.18]{andersonfuller}. 
\end{proof}

\begin{lemma}\label{lemsfrin} Let $A$ be a ring with local units and $I$ a left ideal of $A$. The following conditions are equivalent:

\begin{itemize}
\item[(1)] The left $A$-module $A/I$ is flat;

\item[(2)] For every $x\in I$, $x\in xI$.
\end{itemize}  
\end{lemma}

\begin{proof} Consider the short exact sequence $0\xra I\xra A\xra A/I\xra 0$ in $A\-\MOD$. Suppose that (1) holds. 
For every $x\in I$, $xA$ is a right ideal of $A$. By Lemma \ref{iff}, we have $xA\cdot I=I\cap (xA\cdot A)$. Since $A$ is a ring with local units, for $x\in I\subseteq A$, there exists an idempotent $e\in A$ such that $x=xe=xee$, implying $x\in I\cap (xA\cdot A)$. Thus $x\in xA\cdot I\subseteq xI$, implying that (2) holds. Conversely, suppose that (2) holds. Take any right ideal $J$ of $A$. By Lemma \ref{iff}, we need to show that $JI=I\cap JA$. Obviously, $JI\sub I\cap JA$. On the other hand, for any  $x\in I\cap JA$, we have $x\in I\cap J$. Since $x\in xI$, there exists $y\in I$ such that $x=xy\in JI$. Thus $I\cap JA\sub JI$. 
\end{proof}

As a direct consequence of Lemma~\ref{lemsfrin} we have the following theorem which gives a characterisation of left 
SF-rings. This  will be used in Proposition~\ref{hyhygt} to give a necessary condition for a Steinberg algebra to be a SF-ring. 

\begin{thm} \label{sfchar}
Let $A$ be a ring with local units. Then $A$ is a left SF-ring if and only if for any left maximal ideal $I$ of $A$ and $a\in I$, there exists an element $b\in I$ such that $ab=a$. 
\end{thm}


\subsection{Von Neumann regular rings}
Von Neumann regular rings constitute an important class of rings. A ring $A$ (possibily without unit) is \emph{von Neumann regular} (or regular for short), if for any $x \in A$, we have $x \in xAx$. There are several equivalent module theoretical definitions for von Neumann regular rings when $A$ is a unital ring, such as $A$ is regular if and only if any module over $A$ is flat. Goodearl's book~\cite{goodearlbook} is devoted to this class of rings. In order to extend statements of unital regular rings to the setting of rings with local units, we need to recall quasi-free modules over a ring with local units,  which play the same role as free modules over a ring with unit.

We call an $A$-module \emph{quasi-free} if it is isomorphic to a direct sum of modules of the form $Ae$ with $e^2 = e\in A$. With this definition there are analogous results as for free modules over rings with unit, such as:

\begin{itemize} 
\item[(1)] An $A$-module is in $A\-\MOD$ if and only if it is an image of a quasi-free $A$-module.

\item[(2)] A module in $A\-\MOD$ is finitely generated if and only if it is an image of a finitely generated, quasi-free $A$-module.

\item[(3)] A module in $A\-\MOD$ is (finitely generated) projective in $A\-\MOD$ if and only if it is a direct summand of a (finitely generated) quasi-free $A$-module.
\end{itemize}

These facts will be used in Proposition~\ref{pkhti1}. Next, we will show that for a finitely generated quasi-free $A$-module $F$, the unital ring $\End_A(F)$ is regular. We need the following lemma.

\begin{lemma} \cite[Lemma 1.6]{goodearlbook} \label{lemregular} Let $A$ be a ring with unit. Let $e_1, \cdots, e_n$ be orthogonal idempotents in $A$ such that $e_1+\cdots+e_n=1$. Then $A$ is regular if and only if for any $x\in e_iAe_j$, there exists $y\in e_jAe_i$ such that $xyx=x$.
\end{lemma}

\begin{lemma} \label{endregular}Let $A$ be a regular ring with local units and $F=\bigoplus_{i=1}^n Af_i$ with $f_i^2=f_i\in A$. Then ${\rm End}_{A}(F)$ is regular.  
\end{lemma}

\begin{proof}  Recall that $\varphi: \Hom_A(Ae, M)\xra eM, f\mapsto f(e)$ is an isomorphism with $e\in A$ an idempotent. Note that $\varphi$ is surjective, since for any $em\in eM$ there exists $g:Ae\xra M, ae\mapsto aem$ such that $\varphi(g)=em$. It follows that 
\begin{equation}\label{bgbg}
\End_A(F)\cong 
\begin{pmatrix}
f_1Af_1& f_1Af_2&\cdots &f_1Af_n\\
f_2Af_1&f_2Af_2&\cdots &f_2Af_n\\
\vdots&\vdots&&\vdots\\
f_nAf_1&f_nAf_2&\cdots &f_nAf_n\\
\end{pmatrix}
\end{equation}
 as rings.  Denote the matrix ring in (\ref{bgbg}) by $A'$ and $e_i$ the elementary matrix with $f_i$ in $(i, i)$-th position and all other entries zero. Then $e_1+\cdots+e_n=1_{A'}$, where $1_{A'}$ is the identity for the matrix ring $A'$. Since $A$ is regular, for any $x\in f_iAf_j$ there exists $y\in f_jAf_i$ such that $xyx=x$. It follows that for any $x'\in e_iA'e_j$ there exists $y'\in e_jA'e_i$ such that $x'y'x'=x'$. By Lemma \ref{lemregular}, $\End_A(F)\cong A'$ is regular.
\end{proof}

We can now extend Theorem 1.11 of \cite{goodearlbook} to rings with local units. 

\begin{lemma}\label{summand} Let $M$ be a projective module over a regular ring $A$ with local units. Then any finitely generated submodule of $M$ is a direct summand of $M$.
\end{lemma}
\begin{proof} There exists a quasi-free module $F=\bigoplus_{i\in I} Ae_i$ over $A$ with $I$ a set which contains $M$ as a direct summand. Given any finitely generated submodule $K$ of $M$, we claim that $F$ has a finitely generated quasi-free direct summand $G$ which contains $K$. Suppose that $k_1, \cdots, k_n$ are generators for $K$. We can write them as expressions $k_i=\sum_{j} a_{ij}e_{ij}$ in $F$ with finitely many nonzero $a_{ij}\in A$ and $ij\in I$. We collect all the idempotents $e_{ij}$'s which appear in the expressions of all the $k_i$'s. We may set $G=\bigoplus_{i} \bigoplus_{j} Ae_{ij}$, which is a finitely generated quasi-free direct summand of $F$ containing $K$. It suffices to prove that $K$ is a direct summand of $G$, then $K$ is a direct summand of $M$, since $K$ is contained in $M$.

In order to prove that $K$ is a direct summand of $G$, we take $H=G\bigoplus F'$, where $F'$ is $n$-copies of the quasi-free module $Ae$ with $e\in A$ an idempotent such that $ek_i=k_i$ for all $i=1,\cdots, n$. Observe that there exists a surjective $A$-module homomorphism from $F'$ to $ K$ such that $ae$ of the $i$-th copy is sent to $aek_i$ for $a\in A$ and $i=1, \cdots, n$. Thus we have a surjective $A$-module homomorphism from $H$ to $K$.
It follows that there exists $f\in \End_A(H)$ such that $fH=K$. By Lemma \ref{endregular}, $\End_A(H)$ is regular. Then there exists $g\in \End_A(H)$ such that $fgf=f$. Consequently, $fg$ is an idempotent in $\End_A(H)$ such that $fgH=fH=K$, whence $K$ is a direct summand of $H$. Since $K$ is contained in $G$, $K$ is a direct summand of $G$. 
\end{proof}

Now we state the equivalent conditions for a von Neumann regular ring with local units.

\begin{prop} \label{pkhti1}
Let $A$ be a ring with local units. The following statements are equivalent. 

\begin{enumerate}[\upshape(1)]
\item  $A$ is a von Neumann regular ring;

\item Any finitely generated left (right) ideal of $A$ is generated by an idempotent;

\item   Any left (right) $A$-module is flat. 
\end{enumerate}
\end{prop}
\begin{proof} First note that for a ring $A$ with local units, $a\in Aa$, for any $a\in A$. In the presence of Proposition~\ref{hytrewe} and Lemma~\ref{charflat}, 
the proofs for $(1)\Leftrightarrow(2)$ and $(3)\Rightarrow(1)$  are similar to the case of rings with unit~\cite[Theorem~1.1, Corollary~1.13]{goodearlbook} and are omitted here.

$(1)\Rightarrow(3)$ Suppose that $A$ is regular.  Since any module in $A\-\MOD$ is an image of a quasi-free $A$-module, if we show that 
for a quasi-free left $A$-module $F$ and a submodule $M$ of $F$, $F/M$ is flat then (3) follows. 
 Suppose $K$ is a finitely generated submodule of $M$. By Lemma \ref{summand}, $K$ is a direct summand of $F$ and thus $F/K$ is projective. It is easy to see that $F/M$ is a direct limit of the flat modules $F/K$, where $K$ ranges over all the finitely generated submodules of $M$. Now thanks to Lemma \ref{charflat} (1) and (4), $F/M$ is a flat module in $A\-\MOD$. \end{proof}

\section{Simple flat property for Steinberg algebras}
\label{section3}

In this section, we give a necessary condition for a Steinberg algebra associated to an ample Hausdorff groupoid with coefficients in a field to be a left SF-ring. Specialising to the graph groupoid case, we show that for Leavitt path algebras over a field SF-rings coincide with von Neumann regular rings. We briefly recall the notion of topological groupoids and Steinberg algebras. 

\subsection{Graded groupoids}\label{grogth33}

A groupoid is a small category in which every morphism is invertible. It can also be
viewed as a generalization of a group which has a partial binary operation.  Let $\mG$ be a
groupoid. If $x\in\mG$, $d(x)=x^{-1}x$ is the \emph{domain} of $x$ and $r(x)=xx^{-1}$ is
its \emph{range}. The pair $(x,y)$ is composable if and only if $r(y)=d(x)$. The set
$\mG^{(0)}:=d(\mG)=r(\mG)$ is called the \emph{unit space} of $\mG$. Elements of
$\mG^{(0)}$ are units in the sense that $xd(x)=x$ and $r(x)x=x$ for all $x \in \mG$. For
$U,V\in\mG$, we define
\[
    UV=\big \{\a\b \mid \a\in U,\b\in V \text{ and } r(\b)=d(\a)\big\}.
\]

A topological groupoid is a groupoid endowed with a topology under which the inverse map
is continuous, and such that composition is continuous with respect to the relative
topology on $\mG^{(2)} := \{(\b,\g) \in \mG \times \mG: d(\b) = r(\g)\}$ inherited from
$\mG\times \mG$. An \emph{\'etale} groupoid is a topological groupoid $\mG$ such that the
domain map $d$ is a local homeomorphism. In this case, the range map $r$ is also a local
homeomorphism. An \emph{open bisection} of $\mG$ is an open subset $U\subseteq \mG$ such
that $d|_{U}$ and $r|_{U}$ are homeomorphisms onto an open subset of $\mG^{(0)}$. We say
that an \'etale groupoid $\mG$ is \emph{ample} if there is a basis consisting of compact
open bisections for its topology.


Let $\G$ be a discrete group  with identity
$\varepsilon$ and $\mG$ a topological groupoid. A $\G$-grading of $\mG$ is
a continuous function $c : \mG \xra \G$ such that $c(gh)=c(g)c(h)$ for all $(g,h)
\in \mG^{(2)}$; such a function $c$ is called a continuous $1$-\emph{cocycle} on $\mG$. In this case, we call $\mG$ a \emph{$\G$-graded groupoid} and we write $\mG=\bigsqcup_{\g\in \G}\mG_{\g}$, where $\mG_\g=c^{-1}(\g)$. Note that $\mG_\g \mG_\beta \subseteq \mG_{\g\beta}$. For $\g
\in \G$, we say that $X\subseteq \mG$ is $\g$-graded if $X\subseteq \mG_\g$. We
always have $\mG^{(0)} \subseteq \mG_\varepsilon$, so $\mG^{(0)}$ is
$\varepsilon$-graded. We write $B^{\rm co}_{\g}(\mG)$ for the collection of all
$\g$-graded compact open bisections of $\mG$ and
\[
B_{*}^{\rm co}(\mG)=\bigcup_{\g\in\G} B^{\rm co}_{\g}(\mG).
\]

A subset $U$ of the
unit space $\mG^{(0)}$ of $\mG$ is \emph{invariant} if $d(\g)\in U$ implies $r(\g)\in U$;
equivalently,
\[
    r(d^{-1}(U))=U=d(r^{-1}(U)).
\] Given an element $u\in \mG^{(0)}$, we denote by $[u]$ the smallest invariant subset of
$\mG^{(0)}$ which contains $u$. Then $$r(d^{-1}(u))=[u]=d(r^{-1}(u)).$$  We call $[u]$ an
\emph{orbit}. We denote $\mG_u=\{\g\in\mG \mid d(\g)=u\}$ and  $\mG^u=\{\g\in\mG \mid r(\g)=u\}$ and 
$\mG_u^v=\mG_u \cap \mG^v$.
The \emph{isotropy group} at a unit $u$ of $\mG$ is the group  $\mG_u^u =\{\g\in\mG \mid
d(\g)=r(\g)=u\}.$ A unit $u\in\mG^{(0)}$ is called \emph{$\Gamma$-aperiodic} if $\mG_u^u
\subseteq \mG_\varepsilon$, otherwise $u$ is called \emph{$\Gamma$-periodic}. For an
invariant subset $W\subseteq \mG^{(0)}$, we denote by $W_{\rm ap}$ the collection of
$\Gamma$-aperiodic elements of $W$ and by $W_{\rm p}$ the collection of $\Gamma$-periodic
elements of $W$. Then $W=W_{\rm ap} \bigsqcup W_{\rm p}.$
If $W=W_{\rm ap}$, we say that $W$ is \emph{$\Gamma$-aperiodic}; If $W=W_{\rm p}$, we say
that $W$ is \emph{$\Gamma$-periodic}.


\subsection{Steinberg algebras}\label{subsetion32}

Steinberg algebras were introduced in~\cite{st} in the context of discrete inverse
semigroup algebras and independently in \cite{cfst} as a model for Leavitt path algebras.

Let $\mG$ be an ample Hausdorff topological groupoid. Suppose that $R$ is a unital commutative ring. Consider $A_R(\mG) = C_c(\mG, R)$, the space of
compactly supported continuous functions from $\mG$ to $R$ with $R$ given the discrete topology.
Then $A_R(\mG)$ is an $R$-algebra with addition defined point-wise and
multiplication is given by convolution
$$(f*g)(\g) = \sum_{\g=\a\b}f(\a)g(\b).$$
It is useful to note that $$1_{U}*1_{V}=1_{UV}$$ for compact open bisections $U$ and $V$
(see \cite[Proposition 4.5(3)]{st}). With this structure, $A_R(\mG)$ is an algebra called the \emph{Steinberg algebra}  associated to $\mG$. The algebra $A_R(\mG)$ can also be realised as the span of characteristic functions of the form $1_U$ where $U$ is a compact open bisection (see \cite[Lemma 3.3]{cfst}). 

By \cite[Lemma 2.2]{cm} and \cite[Lemma
3.5]{cfst}, every element $f\in A_{R}(\mG)$ can be expressed as
\begin{equation}
\label{express}
f=\sum_{U\in F}a_{U}1_{U},
\end{equation}
where $F$ is a finite subset of mutually disjoint elements of $B_{*}^{\rm co}(\mG)$.

Recall from \cite[Lemma 3.1]{cs} that if $\mG=\bigsqcup_{\g\in G}\mG_\g$ is a $G$-graded groupoid, then the Steinberg algebra $A_R(\mG)$ is a $\G$-graded
algebra with homogeneous components
\begin{equation}\label{hgboat}
    A_{R}(\mG)_{\g} = \{f\in A_{R}(\mG)\mid \supp(f)\subseteq \mG_\g\}.
\end{equation}
The family of all idempotent elements of $A_{R}(\mG^{(0)})$ is a set of local units for
$A_{R}(\mG)$ (\cite[Lemma 2.6]{cep}). Here, $A_{R}(\mG^{(0)})\subseteq A_{R}(\mG)$ is a
subalgebra. Note that any ample Hausdorff
groupoid admits the trivial cocycle from $\mG$ to the trivial group $\{\varepsilon\}$,
which gives rise to a trivial grading on $A_{R}(\mG)$.

\subsection{Regularity of Steinberg algebras} \label{subsection33}

In this section we give necessary conditions for a Steinberg algebra to be von Neumann regular. The complete characterisation of regular Steinberg algebras is remained to be determined.

For an ample Hausdorff groupoid $\mG$, and a commutative unital ring $R$, we denote by $R\mG$ the (classical) groupoid ring and by $R\mG_u^v\subseteq R\mG$, the set of finite sums $\sum r_i g_i$, where $g_i \in \mG_u^v$ and $r_i\in R$.  Note that $R\mG_v^w R\mG_u^v \subseteq R\mG_u^w$, for $u,v,w \in \mG^{(0)}$. In particular $R\mG_u^u$ is the group ring on the isotropy group $\mG_u^u$. For $u,v\in \mG^{(0)}$ define the map 
\begin{align*}
\phi_u^v:A_R(\mG) \longrightarrow & R\mG_u^v\\
f \longrightarrow & f|_{\mG_u^v}.
\end{align*} Since the support of $f\in A_R(\mG)$ is compact and $\mG_u$ is closed in $\mG$ and has an induced discrete topology, $\supp( f|_{\mG_u^v})$ is a finite set and thus can be represented in the $ R\mG_u^v$. On the other hand, define a (non-canonical) map 
\[\psi_u^v: R\mG_u^v \rightarrow A_R(\mG),\]
as follows: For a finite sum $x=\sum r_i g_i \in R\mG_u^v$, where $g_i\in \mG_u^v$, consider disjoint compact open bisections $V_i$ containing $g_i$ and the element $f=\sum_i r_i 1_{V_i}\in A_R(\mG)$, and define $\psi_u^v(x)=f$. Note that for $x\in R\mG_u^v$
\begin{equation}\label{hfytdwp}
\phi_u^v \psi_u^v (x)=x.
\end{equation}

Observe that if $f\in A_R(\mG)$ such that $\supp(f) \cap \mG_u \subseteq \mG_u^v$, then for any $g\in A_R(\mG)$ and $w\in \mG^{(0)}$, we have 
\begin{equation}\label{hnghtu1}
\phi_u^w(g*f)=\phi_v^w(g)\phi_u^v(f).
\end{equation}
Similarly, 
for $f\in A_R(\mG)$ such that $\supp(f) \cap \mG^u \subseteq \mG_v^u$, then for any $g\in A_R(\mG)$ and $w\in \mG^{(0)}$, we have 
\begin{equation}\label{hnghtu2}
\phi_w^u(f*g)=\phi_v^u(f)\phi_w^v(g).
\end{equation}

\begin{prop}\label{htrbdgew}
Let $\mG$ be an ample Hausdorff groupoid and $R$ a commutative ring with unit. If $A_R(\mG)$ is regular then  for any $x \in R\mG_u^v$ there is a $y\in R\mG_v^u$ such that $xyx=x$. 
\end{prop}
\begin{proof}
Let $x \in R\mG_u^v$. Then $f=\psi_u^v(x) \in  A_R(\mG)$. Since $A_R(\mG)$ is regular we have $f*g*f=f$, for some $g\in A_R(\mG)$.  Thus  
\[\phi_u^v(f*g*f)=\phi_u^v(f).\]
Notice that by the construction of $\psi_u^v$, we have $\supp(f) \cap \mG_u \subseteq \mG_u^v$ and $\supp(f) \cap \mG^v \subseteq \mG_u^v$. Now using (\ref{hnghtu1}) and (\ref{hnghtu2}) we can write
\[\phi_u^v(f)\phi_v^u(g)\phi_u^v(f)=\phi_u^v(f).\]
By~(\ref{hfytdwp}) $\phi_u^v(f)=\phi_u^v(\psi_u^v(x))=x$ and $y:=\phi_v^u(g)\in R\mG_v^u$. Substituting these in the above equation, we obtain $xyx=x$.
\end{proof}

Based on the work of Auslander, Connell~\cite[Theorem~3]{connell} gave a complete characterisation of regular group rings as follows:  A group ring $RG$ is regular if and only if $R$ is regular, $G$ is locally finite and the order of any finite subgroup of $G$ is a unit in $R$. Combining this with Proposition~\ref{htrbdgew} we obtain a necessary conditions for a Steinberg algebra to be regular. 

\begin{cor}\label{chrisrea}
Let $\mG$ be an ample groupoid and $R$ a commutative ring with unit. If $A_R(\mG)$ is regular, then $R$ is regular, $\mG_u^u$ are locally finite and $1\in R$ is divisible by the order of any finite subgroup of $\mG_u^u$, $u \in \mG^{(0)}$. 
\end{cor}
\begin{proof}
By  Proposition~\ref{htrbdgew} if $A_R(G)$ is regular, then $R\mG_u^u$ is regular for any $u \in \mG^{(0)}$. Now by the result of Connell, the corollary follows.  
\end{proof}

We can use this to recover one direction of Abram-Rangaswamy theorem of regularity of Leavitt path algebras~\cite{abramsranga} (see the statement after Corollary~\ref{khtfgtrdd}).

\subsection{SF property for Steinberg algebras}
\label{subsection34}

Throughout this subsection we work with a Steinberg algebras with coefficients in a field $R$. The main result of this section is Proposition~\ref{hyhygt} which gives a necessary condition for Steinberg algebras to be SF-rings. In order to do this, we need to consider certain simple modules of Steinberg algebras which was first considered in \cite{ahls} in the graded setting.

 For $u \in \mG^{(0)}$, let $R[u]$ be the free $R$-module with the basis the set $[u]$, where $[u]$ is the orbit of $u$ (see~\S\ref{grogth33}).  There is a function $f_U:\mG^{(0)}\xra R[u]$ for any compact open bisection $U$ of $\mG$ such that the support of $f_U$ is contained in $d(U)\cap [u]$ and for $w=d(x)\in d(U)\cap [u]$ with $x\in U$, $f_U(w)=r(x)$. This uniquely induces an action of $A_R(\mG)$ on $R[u]$ such that $1_U(w)=f_U(w)$ for $w\in [u]$ which makes $R[u]$ a left $A_R(\mG)$-module (see \cite[Proposition 4.3]{bcfs}).  We show that $R[u]$ is a simple left $A_R(\mG)$-module. Suppose that $V\neq 0$ is a $A_{R}(\mG)$-submodule of $R[u]$. Take a nonzero element $x\in V$. Fix nonzero
elements $r_i \in R$ and pairwise distinct $u_i \in [u]$ such that
$x=\sum_{i=1}^{m}r_{i}u_{i}$. Since $\mG^{(0)}$ is Hausdorff, there exist disjoint open subsets $X_{i}$ of
$\mG^{(0)}$ such that $u_{i}\in X_{i}$ for all $i$. Since $\mG$ is ample, we can choose
compact open sets $B_{i}\subseteq X_{i}$ such that $u_{i}\in B_{i}$, for all $i=1, \cdots,
m$. Now
\begin{equation*}
1_{B_{1}}\cdot x
    = 1_{B_{1}}\cdot \sum_{i=1}^{m}r_{i}u_{i}
    = \sum_{i=1}^{m}r_{i}(1_{B_{1}}\cdot u_{i})
    = r_{1}u_1.
\end{equation*}
Thus $u_1 \in V$ as $R$ is a field. Fix $v \in [u]$ and
choose $x\in\mG$ such that $d(x)=u_1$ and $r(x)=v$. Fix a compact open bisection $D$
containing $x$. Then $1_{D}\cdot u_1=f_{D}(u_1)=r(x)=v\in V$, giving $V=R[u]$. Thus
$R[u]$ is simple. 

We are in a position to state the main result of this section. We require our groupoid to be graded by a totally ordered group. In most applications the grade group is a torsion free abelian group. By~\cite{levi} an abelian group can be equipped with a total order if and only if it is torsion-free.

\begin{prop} \label{hyhygt}
Let $\mG$ be a $\G$-graded ample groupoid, where $c:\mG\rightarrow \Gamma$ is the cocycle, $R$ a field and $\G$ is a totally ordered group. If $A_R(\mG)$ is a left SF-ring, then the unit space $\mG^{(0)}$ of $\mG$ is aperiodic.  
\end{prop}
\begin{proof}
Let $u \in \mG^{(0)}$ and consider the orbit $[u]$ and the free $R$-module $R[u]$.  Consider the exact sequence of left $A_R(\mG)$-modules
\begin{align*} 
A_R(\mG) \stackrel{\phi}{\longrightarrow} & R[u] \longrightarrow 0,\\
1_V \longmapsto & 1_V . u
\end{align*}
 Since $R[u]$ is a simple $A_R(\mG)$-module, $I:=\ker(\phi)$ is a left maximal ideal of $A_R(\mG)$. Suppose $u\in \mG^{(0)}$ is a periodic element in the unit space.  Then there is $x\in \mG_u^u$ with $c(x)\not = \varepsilon$.  Without loss of generality we can assume that $c(x)=\alpha>\varepsilon.$ Since $\mG$ is ample, there is a compact open bisection $V\subseteq \mG_\alpha$ containing $x$ and a compact open set $U \subseteq \mG^{(0)}$ containing $u$.  By replacing $V$ by $UVU$ we can assume that $x\in V \subseteq \mG_\alpha$ and $s(V), r(V) \subseteq U$.   We consider the element $1_U -1_V \in A_R(\mG)$. Then 
\[(1_U-1_V)(u)=1_U(u) -1_V(u)=0.\] Thus $1_U-1_V \in I$. Since $A_R(\mG)$ is a SF-ring, by Theorem~\ref{sfchar},  $(1_U-1_V)b =1_U-1_V$, for some $b\in I$. Writing $b=\sum_i r_i1_{W_i}$, where $W_i$'s are disjoint homogeneous compact open bisections and $r_i$'s are nonzero in $R$, we have
\[1_U-1_V=(1_U-1_V)\sum_i r_i1_{W_i}=(1_U-1_V)1_U\sum_i r_i1_{W_i}= (1_U-1_V) \sum_i r_i1_U1_{W_i}=(1_U-1_V) \sum_i r_i 1_{UW_i}.\] Since $I$ is a left ideal of $A_R(\mG)$, $1_U\sum_i r_i1_{W_i} \in I$, thus without loss of generality we can assume that $r(W_i)\subseteq U$ for all $i$ and 
\[(I_U-1_V) \sum_i r_i1_{W_i}=1_U-1_V.\] It follows that 
\begin{equation}\label{compa}
\sum_i r_i1_{W_i}-\sum_i r_i1_{VW_i}=1_U-1_V.
\end{equation}
We first show that $c(W_i)\geq \varepsilon$ for all $i$.  If this is not the case, choose an $W_l$ with the smallest degree $c(W_l)$. If $VW_l \not = \emptyset$ then $c(VW_l) > c(W_l)$ as $c(V)>\varepsilon$ by our assumption. Then comparing the both sides of \eqref{compa}, the left hand side has negative degrees whereas the right hand side are non-negative. This is a contradiction, so $c(W_i)\geq \varepsilon$ for all $i$. Collecting $W_l$'s with $c(W_l)=\varepsilon$, by comparison in \eqref{compa}, we have 
$\sum_{c(W_l)=\varepsilon} r_l 1_{W_l}= 1_U$. Since $W_l$'s are disjoint, it follows that $W_l \subseteq U$ and $r_l=1$. Thus 
\[b=\sum_i r_i1_{W_i}=\sum_{c(w_i)>\varepsilon} r_i1_{W_i}+1_U.\] 
Substituting this into \eqref{compa}, we get
\[\sum_{c(w_i)>\varepsilon} r_i1_{W_i} - \sum_{c(w_i)>\varepsilon} r_i1_{VW_i}=0    .\] 
Again if $VW_i \not = \emptyset$ then  $c(VW_i) > c(W_i)$. Choosing $W_l$'s with the smallest degree, we have 
$\sum_l r_l 1_{W_l} =0$ which gives that $r_l=0$. This is a contradiction. This forces $b=\sum_{c(W_l)=\varepsilon} r_l 1_{W_l}= 1_U\in I$. But $1_U \not \in \ker(\phi)=I$. Therefore $u$ can not be periodic and thus the unit space $\mG^{(0)}$ of $\mG$ is aperiodic.  
\end{proof}

We are in a position to specialise to the case of Leavitt path algebras. We briefly recall the groupoid of a graph. Let $E=(E^{0}, E^{1}, r, s)$ be a directed graph (see~\cite{lpabook} for conventions on graphs). We denote by $E^{\infty}$ the set of
infinite paths in $E$ and by $E^{*}$ the set of finite paths in $E$. Set
\[
X := E^{\infty}\cup  \{\mu\in E^{*}  \mid   r(\mu) \text{ is not a regular vertex}\}.
\]
Let
\[
\mG_{E} := \big \{(\a x,|\a|-|\b|, \b x) \mid   \a, \b\in E^{*}, x\in X, r(\a)=r(\b)=s(x)\big\}.
\]
We view each $(x, k, y) \in \mG_{E}$ as a morphism with range $x$ and source $y$. The
formulas $(x,k,y)(y,l,z)= (x,k + l,z)$ and $(x,k,y)^{-1}= (y,-k,x)$ define composition
and inverse maps on $\mG_{E}$ making it a groupoid with $\mG_{E}^{(0)}=\{(x, 0, x) \mid
x\in X\}$. One can show that $\mG_E$ is a $\Z$-graded ample Hausdorff groupoid and establish a $\Z$-graded isomorphism $L_R(E) \cong A_R(\mG_E)$ (\cite[Example 3.2]{cs}).

Specialising Proposition~\ref{hyhygt} to the graph groupoid case allows us to answer Ramamurthi's question in positive for Leavitt path algebras.  

\begin{cor}\label{khtfgtrdd} Let $E$ be any directed graph and $R$ a field. If the Leavitt path algebra $L_R(E)$ is a left SF-ring, then $L_R(E)$ is regular.
\end{cor}
\begin{proof}
Since $L_R(E) \cong A_R(\mG_E)$, where $\mG_E$ is a $\Z$-graded groupoid, if $L_R(E)$ is a left SF-ring, then by Proposition~\ref{hyhygt}, $\mG_E^{(0)}$ has to be aperiodic. This immediately implies that $E$ has no cycle. Thus by \cite[Theorem 1]{abramsranga}, $L_R(E)$ is regular. 
\end{proof}

Using Corollary~\ref{chrisrea} one can immediately recover one direction of the Abram-Rangaswamy theorem of regularity of Leavitt path algebras~\cite{abramsranga}.  Namely suppose that for an arbitrary graph $E$ and a commutative ring $R$, $L_R(E)$ is regular. Since $L_R(E)=A_R(\mG_E)$, then by Corollary~\ref{chrisrea}, $R$ is regular and $E$ has no cycle, otherwise for the infinite path $u=cc\cdots$, where $c$ is a cycle, we have $\mG_u^u\cong \Z$, which is a contradiction.

\section{Acknowledgements}
The authors would like to acknowledge Australian Research Council grants DP160101481. They would like to thank Ardeline Mary Buhphang (North Eastern Hill University, Shillong) who brought to their attention the question of whether SF Leavitt path algebras are von Neumann regular. The first author would like to acknowledge the support by SERB Overseas Postdoctoral Fellowship [SB/OS/PDF-317/2016-17], Department of Science and Technology, Govt. of India during this work. She would also like to thank Western Sydney University for their hospitality.

\end{document}